\documentclass[12pt,leqno,draft]{article}
\usepackage{amsfonts}
\usepackage{amsmath, amsthm, amsfonts, amssymb, color}
\usepackage{mathrsfs}
\usepackage{color}
\usepackage{stmaryrd}
\makeatletter
\newcommand{\Spvek}[2][r]{%
  \gdef\@VORNE{1}
  \left(\hskip-\arraycolsep%
    \begin{array}{#1}\vekSp@lten{#2}\end{array}%
  \hskip-\arraycolsep\right)}

\def\vekSp@lten#1{\xvekSp@lten#1;vekL@stLine;}
\def\vekL@stLine{vekL@stLine}
\def\xvekSp@lten#1;{\def\temp{#1}%
  \ifx\temp\vekL@stLine
  \else
    \ifnum\@VORNE=1\gdef\@VORNE{0}
    \else\@arraycr\fi%
    #1%
    \expandafter\xvekSp@lten
  \fi}
\makeatother

\newtheorem{thm}{Theorem}[section]

\newtheorem{lem}[thm]{Lemma}

\newtheorem{rem}[thm]{Remark}
\theoremstyle{definition}

\newcommand{\scr}[1]{\mathscr #1}
\definecolor{wco}{rgb}{0.5,0.2,0.3}

\numberwithin{equation}{section} \theoremstyle{remark}

\newcommand{\ua}{\uparrow}

\title{{\bf    Harnack and Super Poincar\'{e} Inequalities for Generalized Cox-Ingersoll-Ross Model }\footnote{Supported in
 part by  NNSFC (11801406).}
}
\author{
{\bf     Xing Huang $^{a)}$, Fei Zhao$^{a)}$ }\\
\footnotesize{  a) Center for Applied Mathematics, Tianjin University, Tianjin 300072, China}\\
\footnotesize{  zhaofeitju@163.com}}
\begin{document}
\allowdisplaybreaks
\def\R{\mathbb R}  \def\ff{\frac} \def\ss{\sqrt} \def\B{\mathbf
B}
\def\N{\mathbb N} \def\kk{\kappa} \def\m{{\bf m}}
\def\ee{\varepsilon}\def\ddd{D^*}
\def\dd{\delta} \def\DD{\Delta} \def\vv{\varepsilon} \def\rr{\rho}
\def\<{\langle} \def\>{\rangle} \def\GG{\Gamma} \def\gg{\gamma}
  \def\nn{\nabla} \def\pp{\partial} \def\E{\mathbb E}
\def\d{\text{\rm{d}}} \def\bb{\beta} \def\aa{\alpha} \def\D{\scr D}
  \def\si{\sigma} \def\ess{\text{\rm{ess}}}
\def\beg{\begin} \def\beq{\begin{equation}}  \def\F{\scr F}
\def\Ric{\text{\rm{Ric}}} \def\Hess{\text{\rm{Hess}}}
\def\e{\text{\rm{e}}} \def\ua{\underline a} \def\OO{\Omega}  \def\oo{\omega}
 \def\tt{\tilde} \def\Ric{\text{\rm{Ric}}}
\def\cut{\text{\rm{cut}}} \def\P{\mathbb P} \def\ifn{I_n(f^{\bigotimes n})}
\def\C{\scr C}   \def\G{\scr G}   \def\aaa{\mathbf{r}}     \def\r{r}
\def\gap{\text{\rm{gap}}} \def\prr{\pi_{{\bf m},\varrho}}  \def\r{\mathbf r}
\def\Z{\mathbb Z} \def\vrr{\varrho} \def\ll{\lambda}
\def\L{\scr L}\def\Tt{\tt} \def\TT{\tt}\def\II{\mathbb I}
\def\i{{\rm in}}\def\Sect{{\rm Sect}}  \def\H{\mathbb H}
\def\M{\scr M}\def\Q{\mathbb Q} \def\texto{\text{o}} \def\LL{\Lambda}
\def\Rank{{\rm Rank}} \def\B{\scr B} \def\i{{\rm i}} \def\HR{\hat{\R}^d}
\def\to{\rightarrow}\def\l{\ell}\def\iint{\int}
\def\EE{\scr E}\def\no{\nonumber}
\def\A{\scr A}\def\V{\mathbb V}\def\osc{{\rm osc}}
\def\BB{\scr B}\def\Ent{{\rm Ent}}\def\3{\triangle}\def\H{\scr H}
\def\U{\scr U}\def\8{\infty}\def\1{\lesssim}\def\HH{\mathrm{H}}
 \def\T{\scr T}
\maketitle

\begin{abstract} In this paper, the Harnack inequalities and super Poincar\'{e} inequality for generalized Cox-Ingersoll-Ross model are obtained. Since the noise is degenerate, the intrinsic metric has been introduced to construct the coupling by change of  measure. By using isoperimetric constant, some optimal estimate of the rate function in the super Poincar\'{e} inequality for the associated Dirichlet form is also obtained.
\end{abstract} \noindent
 \vskip 2cm

\section{Introduction}
The SDE
\begin{equation}\label{E0}
\d X_t=(\alpha-\delta X_t)\d t+\sqrt{X_t} \d B_t,\ \ X_0>0,
\end{equation}
which is called CIR (Cox-Ingersoll-Ross) model \cite[Section 4.6]{C}, is used to characterize the evolution of the
interest rate in finance. In \cite{A,A1,A2,CJM,GR,S,WMC,YW}, the authors investigate the convergence rate of various numerical schemes of \eqref{E0}. Zhang and Zheng \cite{ZZ} obtain the Harnack inequality and super Poincar\'{e} for \eqref{E0}. See \cite{CIR,CIR1,CL} for more introductions on \eqref{E0}.

In this paper, we  consider stochastic differential equations on $[0,\infty)$:
\begin{equation}\label{1.1}
\d X_t=(\alpha-\delta X_t)\d t+X_t^{h} \d B_t,
\end{equation}
with constant $\frac{1}{2}<h<1$, ${\alpha}, {\delta}>0$, and $B_t$ a is one-dimensional Brownian motion on some complete filtration probability space $(\OO, \F, \{\F_{t}\}_{t\ge 0}, \P)$. We call \eqref{1.1} a generalized CIR model. By \cite{IW,KS}, for any $x\in[0,\infty)$, \eqref{1.1} has a unique non-negative strong solution $X_t^x$ with initial value $x$. Let $P_t$ be the associated semigroup, i.e.
$$P_tf(x)=\E f(X_t^x), \ \  f\in\B_b([0,\infty)).$$

Compared with SDE \eqref{E0}, the diffusion in \eqref{1.1} has stronger degeneration on 0 due to $h>\frac{1}{2}$, which leads to worse regularity of the solution. Thus, the Harnack inequality for the semigroup associated to \eqref{1.1} is non-trivial.

Wang \cite{W} introduced coupling by change of measure to establish Harnack inequality in the SDEs with non-degenerate diffusion coefficients, see \cite{ATW,W3,WY} for more models. Wang \cite{W2} also gives some conditions to obtain super Poincar\'{e} inequality. Zhang and Zheng \cite{ZZ} obtained the functional inequalities of \eqref{E0} under some reasonable conditions.

In this paper, we will prove the Harnack and super Poincar\'{e} inequality for \eqref{1.1}, which cover the results in  \cite{ZZ} where $h$ is assumed to be $\frac{1}{2}$.

 The paper is organized as follows: In Section 2, we give main results on Harnack and super Poincar\'{e} inequality, which will be proved in Section 4 and Section 5 respectively; In Section 3, we give some lemmas which will be used in the sequel.
\section{Main Results}
\subsection{Harnack Inequality and Gradient Estimate}
As we know, the intrinsic metric associate to the generator of \eqref{1.1} is defined as
\begin{align}\label{rho}\rho (s,t) = \int_{s \wedge t}^{s \vee t} \frac{\d r}{r^{h}},\ \ s,t\in [0,\infty).
\end{align}
For any $f\in C^1([0,\infty))$, $x\in[0,\infty)$, define
$$\nabla^h f(x):=\lim_{y\to x}\frac{f(y)-f(x)}{\frac{1}{1-h}y^{1-h}-\frac{1}{1-h}x^{1-h}}=\frac{f'(x)}{x^{-h}}=x^hf'(x).$$
$\nabla^h$ is called the intrinsic gradient. Obviously, we have
$$|\nabla^h f(x)|=\lim_{\rho(y,x)\to 0}\frac{|f(y)-f(x)|}{\rho(y,x)}=x^h|f'(x)|.$$
The following theorem gives the result on Harnack inequality and estimate of intrinsic gradient $\nabla^h$.
\begin{thm}\label{T-Har}  Assume  $\frac{1}{2}<h<1$
 and $\alpha \geq \frac{h}{2}$. Then the following assertions hold.
\begin{enumerate}
\item[(1)] The Harnack inequality holds, i.e. for any $T>0$, $p>1$ and $x$, $y \in [0,\infty)$, it holds
 \beg{equation*}\beg{split}  (P_Tf)^p(y)\le  &P_Tf^p(x)
  \exp \left[ {\frac{p(\delta -\frac{h}{2})  (y^{1-h}-x^{1-h})^2  }{(p-1)(1-h)(\e^{2(1-h)(\delta -\frac{h}{2}) T}-1)}}\right], \ f\in \B^+_b([0,\infty)).
  \end{split}\end{equation*}
  Moreover, for any  $f\in \B^+_b([0,\infty))$ with $f>0$, the Log-Harnack inequality
  \beg{equation*}\beg{split}  P_T \log f(y)\le  &\log P_T f (x)
  +{\frac{(\delta -\frac{h}{2})  (y^{1-h}-x^{1-h})^2  }{(1-h)(\e^{2(1-h)(\delta -\frac{h}{2}) T}-1)}}
  \end{split}\end{equation*}
  holds.
\item[(2)] For any $f\in C^1_b([0,\infty)$, the estimate of the intrinsic gradient holds:
\begin{align*}
|\nabla^h P_Tf(x)|
&\leq \e^{-(1-h) (\delta -\frac{h}{2}) T}P_T |\nabla^ h f|(x),\ \ T>0,x\in[0,\infty).
\end{align*}
\end{enumerate}
 \end{thm}
\subsection{Super Poincar\'{e} inequality}
We firstly introduce some notations:
\begin{align*}
 &C(x)=\int_1^x\frac{\alpha-\delta y}{y^{2h}}\d y=\frac{2 \alpha }{1-2h} x^{1-2h}-\frac{\delta }{1-h} x^{2-2h}-\left(\frac{2 \alpha }{1-2h}-\frac{\delta }{1-h}\right),\ \ x>0,\\
 &\Gamma_0=2\e ^{-\left(\frac{2 \alpha }{1-2h}-\frac{\delta }{1-h}\right)},\\
 &Z=\int_0^\infty\frac{\e^{C(x)}}{\frac{1}{2}x^{2h}}\d x= \Gamma_0\int_{0}^{\infty} x^{-2h}\e^{\frac{2 \alpha }{1-2h} x^{1-2h}-\frac{\delta }{1-h} x^{2-2h}} \d x,
  \end{align*}
  and
$$ \mu (\mathrm{d}x)=\frac{\Gamma_0x^{-2h}\e^{\frac{2 \alpha }{1-2h} x^{1-2h}-\frac{\delta }{1-h} x^{2-2h}}}{Z} \d x=: \eta (x) \d x.$$
 Consider second-ordered differential operator on $L^2(\mu)$:
  $$L=\frac{1}{2}x^{2h}\frac{\mathrm{d^{2}}}{\mathrm{d}x^{2}}+(\alpha-\delta x)\frac{\mathrm{d}}{\mathrm{d}x}. $$
 Let $(\mathcal{E}, \mathcal{D}(\mathcal{E}))$ be the associated Dirichlet form to $L$ on $L^2(\mu)$. In particular, we have
$$\mathcal{E}(f,f)=\frac{1}{2} \int_{0}^{\infty} x^{2h}[f' (x)]^{2}\mu (\mathrm{d}x),\ \ f\in C_0^1([0,\infty)).$$
Let $\rho$ be defined in \eqref{rho}, then we have
$$\rho (x,y)=\int_{x \wedge y}^{x \vee y} \frac{\d r}{r^{h}}=\frac{1}{1-h}|y^{1-h}-x^{1-h} |,\ \ x,y\in[0,\infty).$$
For any open set $D \subset [0,\infty) $, the boundary measure of $D$ induced by $\mu$ is defined as $$\mu_{\partial}(\partial D):=\lim\limits_{\varepsilon \to 0} \frac{\mu(D_{\varepsilon })-\mu(D)}{\varepsilon }$$ with $D_{\varepsilon }=\{x \in [0,\infty) | \rho (x,D) \leq \varepsilon\}.$

The isoperimetric constant is defined as
$$k(r):=\inf_{\mu(D) \leq r} \frac{\mu_{\partial}(\partial D)}{\mu(D)},\ \ r>0.$$
\begin{thm}\label{T-Super} Assume $\frac{1}{2}<h<1$. Then the following assertions hold.
\begin{enumerate}
\item[(1)] The super Poincar\'{e} inequality
 \begin{align}\label{sp}\mu(f^2)\leq r\mathcal{E}(f,f)+\beta(r)\mu(|f|)^2, \ \ r>0, f\in \mathcal{D}(\mathcal{E})
 \end{align}
 holds for $\beta (r)=\frac{4}{k^{-1}(2\sqrt{2}r^{-\frac{1}{2}})}$ with $k^{-1}(r)=\sup\{s\geq 0, k(s)>r)\}$.
\item[(2)] Moreover, there exists constants $c,r_0>0$ such that
$k(r)\geq c(-\log r)^{\frac{1}{2}}$ for any $r\in(0,r_0)$. Thus, \eqref{sp} holds with $\beta (r)= \e^{C(1+r^{-1})}$ for some constant $C>0$.
\item[(3)] Finally, $\beta (r)$ in (2) is optimal in the following sense: the super Poincar\'{e} inequality can not hold for any $\beta(r)=\e^{C(1+r^{-\lambda})}$ with $0<\lambda<1$ and some constant $C>0$.
  \end{enumerate}
 \end{thm}
\section{Some Preparations}
In this section, we give two important lemmas which will be used in the proof of Theorem \ref{T-Har}.
\beg{lem}\label{L1} Assume $\frac{1}{2}<h<1$ and SDE
\begin{equation}\nonumber
\d X_t=b(X_t)\d t+X_t^{h} \d B_t, \ \ X_0=x>0
\end{equation}
has a non-explosive and non-negative solution $X_t$. Here, $b:[0,\infty) \rightarrow \R$ is locally bounded and continuous at point $0$, $b(0)>0$. Then $\P$-a.s.
  \begin{equation}\nonumber
  \int_0^\infty I_{\{0\}} (X_t) \d t= 0.
  \end{equation}
 \end{lem}
 \begin{proof}
For any $n\geq 1$, construct function $\varphi_{n}:[0,\infty) \rightarrow \R$ as follows£º
 \begin{eqnarray}\nonumber
\varphi_{n} (x)=
\begin{cases}
\frac{1}{2(h+1)n},&{x\geq \frac{1}{n} }, \\
 \frac{n^{2h+1}}{2(h+1)}(\frac{1}{n}-x)^{2(h+1)}+\frac{1}{2(h+1)n},&{0 \leq x<\frac{1}{n}}.\\
\end{cases}
\end{eqnarray}
It is not difficult to see
 \begin{equation}\label{1.2}
| \varphi_{n} (x)| \leq \frac{1}{h+1},\ \ | \varphi'_{n} (x)| \leq 1,\ \ | \varphi''_{n} (x)\cdot x^{2h}|=\frac{2h+1}{n^{2h-1}} \leq 2h+1,
 \end{equation}
 and
 \begin{equation}\label{1.3}
\lim\limits_{n \to \infty} \varphi_{n} (x)=0,\ \ \lim\limits_{n \to \infty}  \varphi'_{n} (x)=-I_{\{0\}} (x),\ \ \lim\limits_{n \to \infty}| \varphi''_{n} (x)\cdot x^{2h}|=0.
\end{equation}
Letting $\tau_{m}$=$\inf$ $\{$ $t$ $\geq$ 0$ : X_t \geq m\}$, since $X_t$ is non-explosive, then we have $\tau_{\infty}:=\lim\limits_{m \to \infty} \tau_{m}=\infty$. Applying It\^{o}'s formula to $\varphi_{n} (X_t)$, we arrive at
\begin{equation}\label{1.4}
\d \varphi_{n} (X_t)=\varphi'_{n} (X_t) b(X_t) \d t+\frac{1}{2} \varphi''_{n} (X_t) X_t^{2h} \d t+\varphi'_{n} (X_t)X_t^{h} \d B_t.
\end{equation}
This implies
\beq\label{1.5}\begin{split}
\varphi_{n} (X_{t \wedge \tau_{m}})&=\varphi_{n} (x)+\int_{0}^{t \wedge \tau_{m}} \varphi'_{n} (X_s) b(X_s) \d s\\
&+\frac{1}{2}\int_{0}^{t \wedge \tau_{m}} \varphi''_{n} (X_s) X_s^{2h} \d s+\int_{0}^{t \wedge \tau_{m}} \varphi'_{n} (X_s)X_s^{h} \d B_s.
\end{split}\end{equation}
Combining the definition of $\tau_m$ and taking expectations in \eqref{1.5}, we obtain
\beq\nonumber
\E \big(\varphi_{n} (X_{t \wedge \tau_{m}})\big)-\varphi_{n} (x)=\E \int_{0}^{t \wedge \tau_{m}} \varphi'_{n} (X_s) b(X_s) \d s+ \frac{1}{2}\E\int_{0}^{t \wedge \tau_{m}} \varphi''_{n} (X_s) X_s^{2h} \d s.
\end{equation}
Thus, \eqref{1.2}-\eqref{1.3} and dominated convergence theorem yield
\begin{align}\label{lim}
\lim\limits_{n \to \infty} \E \int_{0}^{t \wedge \tau_{m}} \varphi'_{n} (X_s) b(X_s) \d s=0.
\end{align}
Since $b$ is locally bounded, there exists $M>0$ such that
$$\big| \varphi'_{n} (x)b(x)\big| \leq \sup_{x \in [0,\frac{1}{n})}\big|b(x)\big| \leq \sup_{x \in [0,1)}\big|b(x)\big| \leq M.$$
Mover, it is clear
$$\lim\limits_{n \to \infty}  \varphi'_{n} (x) b(x)=-I_{\{0\}}(x) b(0).$$
So, this together with dominated convergence theorem and \eqref{lim} implies
  $$\lim\limits_{n \to \infty} \E \int_{0}^{t \wedge \tau_{m}} \varphi'_{n} (X_s) b(X_s) \d s=\E  \int_{0}^{t \wedge \tau_{m}} I_{\{0\}}(X_{s}) b(0) \d s=0,$$
  which yields $\E \int_{0}^{t \wedge \tau_{m}} I_{\{0\}} (X_{s})\d s=0$ due to $b(0)>0$. Letting firstly $t$ goes to $\infty$ and then $m$ tends to $\infty$, we have $\E \int_{0}^{\infty} I_{\{0\}} (X_{s})\d s=0$. Thus, we have $\P$-a.s. $\int_{0}^{\infty} I_{\{0\}} (X_{s}) \d s=0$.
 \end{proof}

 \beg{lem}\label{L2} Let $\frac{1}{2}<h<1$ and $\alpha \geq \frac{h}{2}$. Then for any $x$, $y \in [0,\infty)$ with $x<y$, we have
\begin{equation}\label{1.6}
\alpha \left(\frac{1}{y^{h}}-\frac{1}{x^{h}}\right)+\frac{h}{2} \left(\frac{1}{x^{1-h}}-\frac{1}{y^{1-h}}+x^{1-h}-y^{1-h}\right) \leq 0.
\end{equation}
 \end{lem}
 \begin{proof}
We divide the proof into two cases.
\begin{enumerate}
\item[(1)] Case 1: $0 \leq x<1$.
\begin{enumerate}
\item[(i)] $y<1$. Consider function $w(z)=\frac{1}{z^{1-h}}-\frac{1}{z^{h}}$, $z>0$. The derivative of $w$ is $$w'(z)= \frac{h-(1-h)z^{2h-1}}{z^{1+h}}.$$ Letting $$z_0=\left(\frac{h}{1-h}\right)^{\frac{1}{2h-1}},$$ then we have $w'(z_0)=0$. Noting that $z_0>1$ due to $h\in(\frac{1}{2},1)$, $w$ is strictly increasing on $[0,1)$. Since $0\leq x<y\leq 1$, we obtain $w(x)<w(y)$, i.e. $$\frac{1}{x^{1-h}}-\frac{1}{y^{1-h}}+\frac{1}{y^{h}}-\frac{1}{x^{h}}<0.$$
This together with $\frac{1}{2}<h<1$, $\alpha \geq \frac{h}{2}$ and $x<y$ implies
\begin{align*}
&\alpha \left(\frac{1}{y^{h}}-\frac{1}{x^{h}}\right)+\frac{h}{2} \left(\frac{1}{x^{1-h}}-\frac{1}{y^{1-h}}+x^{1-h}-y^{1-h}\right)\\
&=\left(\alpha -\frac{h}{2}\right)\left(\frac{1}{y^{h}}-\frac{1}{x^{h}}\right)+\frac{h}{2} \left(\frac{1}{x^{1-h}}-\frac{1}{y^{1-h}}+\frac{1}{y^{h}}-\frac{1}{x^{h}}\right)+\frac{h}{2} \left(x^{1-h}-y^{1-h}\right)<0.
\end{align*}
\item[(ii)]  $y>1$. Since $\frac{1}{2}<h<1$, we have $\frac{1}{y^{h}}<\frac{1}{y^{1-h}}$. By the same reason, it holds $\frac{1}{x^{h}}>\frac{1}{x^{1-h}}$ due to $x<1$. Thus, $$ \frac{1}{x^{1-h}}-\frac{1}{y^{1-h}}+\frac{1}{y^{h}}-\frac{1}{x^{h}}<0.$$
Again thanks to $\frac{1}{2}<h<1$, $\alpha \geq \frac{h}{2}$ and $x<y$, we obtain
\begin{align*}
&\alpha \left(\frac{1}{y^{h}}-\frac{1}{x^{h}}\right)+\frac{h}{2} \left(\frac{1}{x^{1-h}}-\frac{1}{y^{1-h}}+x^{1-h}-y^{1-h}\right)\\
&=\left(\alpha -\frac{h}{2}\right)\left(\frac{1}{y^{h}}-\frac{1}{x^{h}}\right)+\frac{h}{2} \left(\frac{1}{x^{1-h}}-\frac{1}{y^{1-h}}+\frac{1}{y^{h}}-\frac{1}{x^{h}}\right)+\frac{h}{2} \left(x^{1-h}-y^{1-h}\right)<0.
\end{align*}
\end{enumerate}
\item[(2)] Case 2: $x \geq 1$. Firstly, we have $y>1$ due to $x<y$. So, we get from $\frac{1}{2}<h<1$ and $x<y$ that
\begin{align*}
&\alpha \left(\frac{1}{y^{h}}-\frac{1}{x^{h}}\right)+\frac{h}{2} \left(\frac{1}{x^{1-h}}-\frac{1}{y^{1-h}}+x^{1-h}-y^{1-h}\right)\\
&\leq \frac{h}{2} \left(\frac{y^{1-h}-x^{1-h}}{x^{1-h}y^{1-h}} +\left(x^{1-h}-y^{1-h}\right)\right)\\
&\leq\frac{h}{2} \left(y^{1-h}-x^{1-h}+\left(x^{1-h}-y^{1-h}\right)\right)=0.
\end{align*}
\end{enumerate}
Thus, we complete the proof.
 \end{proof}
With the above two lemmas in hand, we finish the proof of Theorem \ref{T-Har} below.
\section{Proof of Theorem \ref{T-Har}}
We use the coupling by change of measure to derive the Harnack inequality.
\begin{proof}[Proof of Theorem \ref{T-Har}]

(1) Fix $T>0$. For any $x,y \in [0,\infty)$, without loss of generality, we may assume that $y > x$. Let $X_t$ solve \eqref{1.1}  with $X_0= x$,
and $Y_t$ solve the equation
\beq\label{3.1} \d Y_t=(\alpha-\delta Y_t)\d t+Y_t^{h} \d B_t-I_{[0,\tau)} \xi (t) Y_t^{h} \d t\end{equation}
with
$Y_0= y$, here $$\xi (t) := \frac{2(\delta-\frac{h}{2})  (y^{1-h}-x^{1-h}) \e^{(1-h)(\delta -\frac{h}{2} )t} }{(\e^{2(1-h)(\delta-\frac{h}{2}) T}-1)},\ \ t \geq 0,$$ and $\tau :=\inf \{t \geq 0:X_t=Y_t\}$, which is the coupling time. Let $Y_t=X_t$ for $t\geq \tau$. We will prove $\tau<T$.

For any $\varepsilon > 0$, let $$\rho_{\varepsilon} (s,t) = \int_{s \wedge t}^{s \vee t} \frac{\d r}{(r+ \varepsilon)^{h}},\ \ s,t\in [0,\infty).$$
Applying It\^{o}'s formula to $\rho_{\varepsilon} (X_t,Y_t)$, we have
\begin{equation}\begin{split}\label{3.2}
&\d \rho_{\varepsilon} (X_t,Y_t)\\
&=\frac{\partial \rho_{\varepsilon} (X_t,Y_t)}{\partial x} \d X_t+\frac{\partial \rho_{\varepsilon} (X_t,Y_t)}{\partial y} \d Y_t+\frac{1}{2} \frac{\partial^{2} \rho_{\varepsilon} (X_t,Y_t)}{\partial x^{2}} \d\< X\>_t+\frac{1}{2} \frac{\partial^{2} \rho_{\varepsilon} (X_t,Y_t)}{\partial y^{2}} \d \<Y\>_t\\
&+\frac{\partial^{2} \rho_{\varepsilon} (X_t,Y_t)}{\partial x \partial y} \d \<X,Y\>_t\\
&=-\frac{\d X_t}{(X_t+ \varepsilon)^{h}}+\frac{\d Y_t}{(Y_t+ \varepsilon)^{h}}+\frac{h X_t^{2h} \d t}{2(X_t+ \varepsilon)^{h+1}}-\frac{h Y_t^{2h} \d t}{2(Y_t+ \varepsilon)^{h+1}}\\
&=\left(\frac{Y_t^{h}}{(Y_t+ \varepsilon)^{h}}-\frac{X_t^{h}}{(X_t+ \varepsilon)^{h}}\right) \d B_t- \delta \left((Y_t+ \varepsilon)^{1-h}-(X_t+ \varepsilon)^{1-h}\right) \d t-\frac{ \xi (t) Y_t^{h}}{(Y_t+ \varepsilon)^{h}} \d t\\
&+(\delta \varepsilon +\alpha)\left(\frac{1}{(Y_t+ \varepsilon)^{h}}-\frac{1}{(X_t+ \varepsilon)^{h}}\right) \d t+\frac{h}{2} \left(\frac{X_t^{2h}}{(X_t+ \varepsilon)^{h+1}}-\frac{Y_t^{2h}}{(Y_t+ \varepsilon)^{h+1}}\right) \d t, \ t<\tau.
\end{split}\end{equation}
Combining the definition of $\rho_\varepsilon$, we arrive at
\begin{equation}\begin{split}\label{3.3}
&\d \left[\frac{1}{1-h} \left((Y_t+ \varepsilon)^{1-h}-(X_t+ \varepsilon)^{1-h}\right)\right]\\
&=\left(\frac{Y_t^{h}}{(Y_t+ \varepsilon)^{h}}-\frac{X_t^{h}}{(X_t+ \varepsilon)^{h}}\right) \d B_t- \left(\delta-\frac{h}{2}\right) \left((Y_t+ \varepsilon)^{1-h}-(X_t+ \varepsilon)^{1-h}\right) \d t\\
&-\frac{ \xi (t) Y_t^{h}}{(Y_t+ \varepsilon)^{h}} \d t+
M(X_t,Y_t,\varepsilon) \d t,\ \ t<\tau,
\end{split}\end{equation}
where
\begin{align*}&M(X_t,Y_t,\varepsilon)=(\delta \varepsilon +\alpha)\left(\frac{1}{(Y_t+ \varepsilon)^{h}}-\frac{1}{(X_t+ \varepsilon)^{h}}\right)\\
&+\frac{h}{2}\left[\left(\frac{X_t^{2h}}{(X_t+ \varepsilon)^{h+1}}-\frac{Y_t^{2h}}{(Y_t+ \varepsilon)^{h+1}}\right)\right]-\frac{h}{2}\left((Y_t+ \varepsilon)^{1-h}-(X_t+ \varepsilon)^{1-h}\right).
\end{align*}
It follows from \eqref{3.3} that
\begin{equation}\begin{split}\label{3.4}
&\frac{1}{1-h} \e^{(1-h) (\delta -\frac{h}{2}) (\tau \wedge T)} \left((Y_{\tau \wedge T}+ \varepsilon)^{1-h}-(X_{\tau \wedge T}+ \varepsilon)^{1-h}\right)+ \int_{0}^{\tau \wedge T} \frac{\e^{(1-h) (\delta -\frac{h}{2})t} \xi (t) Y_t^{h}}{(Y_t+ \varepsilon)^{h}} \d t\\
&=\frac{1}{1-h}\left((y+ \varepsilon)^{1-h}-(x+ \varepsilon)^{1-h}\right)+ \int_{0}^{\tau \wedge T} \e^{(1-h) (\delta -\frac{h}{2}) t} \left(\frac{Y_t^{h}}{(Y_t+ \varepsilon)^{h}}-\frac{X_t^{h}}{(X_t+ \varepsilon)^{h}}\right) \d B_t\\
&+\int_{0}^{\tau \wedge T} \e^{(1-h) (\delta -\frac{h}{2}) t} M(X_t,Y_t,\varepsilon) \d t.
\end{split}\end{equation}
Let
\begin{equation}\begin{split}\nonumber
I_{1}:&= \lim\limits_{\varepsilon \to 0} \E \left| \int_{0}^{\tau \wedge T} \e^{(1-h) (\delta -\frac{h}{2}) t} \Big(\frac{Y_t^{h}}{(Y_t+ \varepsilon)^{h}}-\frac{X_t^{h}}{(X_t+ \varepsilon)^{h}}\Big) \d B_t\right|^2\\
      & =\lim\limits_{\varepsilon \to 0} \E  \int_{0}^{\tau \wedge T} \e^{2(1-h) (\delta -\frac{h}{2}) t} \Big(\frac{Y_t^{h}}{(Y_t+ \varepsilon)^{h}}-\frac{X_t^{h}}{(X_t+ \varepsilon)^{h}}\Big)^{2} \d t,
\end{split}\end{equation}
By Lemma \ref{L1} and $Y_s\geq X_s$, we have $\P$-a.s
\begin{equation}\nonumber
\int_{0}^{\infty} I_{\{X_{s}=0\}} \d s=0,\ \ \int_{0}^{\infty} I_{\{Y_{s}=0\}} \d s=0.
\end{equation}
This implies
\begin{equation}\nonumber
I_{1}\leq \E \left ( \int_{0}^{T} \e^{2(1-h) (\delta -\frac{h}{2}) t}(I_{\{Y_{t} \neq 0\}}-I_{\{X_{t} \neq 0\}})^{2} \d t\right )=0.
\end{equation}
Since $X$ and $Y$ are continuous, by dominated convergence theorem and Lemma \ref{L2}, we obtain \begin{align*}&\lim_{\varepsilon\to 0}\int_{0}^{\tau \wedge T} \e^{(1-h) (\delta -\frac{h}{2}) t} M(X_t,Y_t,\varepsilon) \d t\\
&= \int_{0}^{\tau \wedge T}\e^{(1-h) (\delta -\frac{h}{2}) t}\lim_{\varepsilon\to 0} M(X_t,Y_t,\varepsilon) \d t\\
&=\int_{0}^{\tau \wedge T}\left(\alpha\left(\frac{1}{Y_t^{h}}-\frac{1}{X_t^{h}}\right)+\frac{h}{2}\left[\left(\frac{X_t^{2h}} {X_t^{h+1}}-\frac{Y_t^{2h}}{Y_t^{h+1}}\right)-\left(Y_t^{1-h}-X_t^{1-h}\right)\right]\right)\d t\\
&\leq 0.
\end{align*}
Thus, letting $\varepsilon$ go to $0$ in \eqref{3.4}, it holds $\P$-a.s.
\begin{equation}\begin{split}\label{3.5}
&\int_{0}^{\tau \wedge T} \e^{(1-h) (\delta -\frac{h}{2}) t} \xi (t)  \d t+\frac{1}{1-h}\e^{(1-h) (\delta -\frac{h}{2}) (\tau \wedge T)} (Y_{\tau \wedge T}^{1-h}-X_{\tau \wedge T}^{1-h})\leq \frac{1}{1-h} (y^{1-h}-x^{1-h}).
\end{split}\end{equation}
On the other hand, by the definition of $\xi(t)$, it is easy to see
\begin{align}\label{xi}\int_{0}^{\tau \wedge T} \e^{(1-h) (\delta -\frac{h}{2}) t} \xi (t)  \d t=\frac{(y^{1-h}-x^{1-h})(\e^{2(1-h)(\delta -\frac{h}{2}) (\tau \wedge T)}-1)}{(1-h)(\e^{2(1-h)(\delta -\frac{h}{2}) T}-1)}.
\end{align}
This and \eqref{3.5} imply $\P(\tau >T)=0$. In fact, if $\P(\tau >T)>0$, considering \eqref{3.5} on the set $\{\tau>T\}$, we have
$$\frac{1}{1-h} (y^{1-h}-x^{1-h})+\frac{1}{1-h} \e^{(1-h) (\delta -\frac{h}{2}) T} (Y_{T}^{1-h}-X_{T}^{1-h}) \leq \frac{1}{1-h} (y^{1-h}-x^{1-h}).$$ This is impossible, and $\P(\tau \leq T)=1$.

Let $$R = \exp\left[\int_{0}^{\tau} \xi (t)  \d B_t-\frac{1}{2}\int_{0}^{\tau} \xi^{2} (t)  \d t\right].$$
By Girsanov's theorem, under the probability $\d \Q:=R\d \P$, the process$$\widetilde{B_{t}} =B_{t}- \int_{0}^{t} 1_{[0,\tau)}(s)\xi (s)  \d s,\ \ t \geq 0$$ is a one-dimensional Brownian motion. Rewrite the equation for $Y_{t}$ as$$\d Y_t=(\alpha-\delta Y_t)\d t+Y_t^{h} \d \widetilde{B_{t}} ,\ \ Y_0=y.$$
We see that the distribution of $Y$ under $\Q$ coincides with that of $X^{y}$ under $\P$. Moreover,  $\Q$-a.s. $X_T=Y_T$.
Thus, $$P_T f(y)=\E^{\Q}(f(Y_T))=\E^{\Q}(f(X_T))=\E(Rf(X_T)).$$
By H\"{o}lder's inequality, we have
\begin{align}\label{Holder}
(P_T f(y))^{p}\leq (\E(R^{p/(p-1)}))^{p-1} \cdot\E(f^{p}(X_T))=P_T f^{p}(x) \cdot (\E(R^{p/(p-1)}))^{p-1}.
\end{align}
On the other hand, from the definition of $R$ and $\xi(t)$, we arrive at
\begin{align*}
\E(R^{p/(p-1)}) &\leq \exp\left[\frac{p}{2(p-1)^{2}}\int_{0}^{T} \xi^{2} (t)  \d t\right]\\
&\times \E \left(\exp\left[\frac{p}{p-1}\int_{0}^{\tau} \xi (t)  \d B_t-\frac{p^{2}}{2(p-1)^{2}}\int_{0}^{\tau} \xi^{2} (t)  \d t\right]\right)\\
&\leq \exp \left[\frac{p}{2(p-1)^{2}}\int_{0}^{T} \xi^{2} (t)  \d t\right]\\
&=\exp \left[ {\frac{p(\delta -\frac{h}{2})  (y^{1-h}-x^{1-h})^2  }{(p-1)^{2}(1-h)(\e^{2(1-h)(\delta -\frac{h}{2}) T}-1)}}\right] .
\end{align*}
This together with \eqref{Holder} yields
\begin{equation}\nonumber
(P_T f(y))^p\leq (P_T f^p(x)) \exp \left[ {\frac{p(\delta -\frac{h}{2})  (y^{1-h}-x^{1-h})^2  }{(p-1)(1-h)(\e^{2(1-h)(\delta -\frac{h}{2}) T}-1)}}\right],\ \ f\in \B^+_b([0,\infty)).
 \end{equation}
Similarly, we have $$P_T \log f(y)=\E^{\Q}( \log f(Y_T))=\E^{\Q}( \log f(X_T))=\E(R \log f(X_T)).$$
Young's inequality implies  $$\E(R \log f(X_T)) \leq \E(R \log  R)+ \log \ \E(f(X_T))=\E(R \log R)+ \log (P_T f(x)).$$
It is not difficult to  see that
\begin{equation}\begin{split}\nonumber
&\E(R \log  R)=\E^{\Q}\log  R\\
&=\E^{\Q}\left(\int_{0}^{\tau} \xi (t)  \d B_t-\frac{1}{2}\int_{0}^{\tau} \xi^{2} (t)  \d t\right)\\
&=\E^{\Q}\left(\int_{0}^{\tau} \xi (t)  \d \widetilde{B_t}+\int_{0}^{\tau} \xi^{2} (t)  \d t-\frac{1}{2}\int_{0}^{\tau} \xi^{2} (t)  \d t\right)\\
&=\frac{1}{2}\E^{\Q}\left(\int_{0}^{\tau} \xi^{2} (t)  \d t\right)\leq {\frac{(\delta -\frac{h}{2})  (y^{1-h}-x^{1-h})^2  }{(1-h)(\e^{2(1-h)(\delta -\frac{h}{2}) T}-1)}}.
\end{split}\end{equation}
Thus, the log-Harnack inequality holds, i.e.
$$P_T \log f(y)\leq \log P_T f(x)) +{\frac{(\delta -\frac{h}{2})  (y^{1-h}-x^{1-h})^2  }{(1-h)(\e^{2(1-h)(\delta -\frac{h}{2}) T}-1)}},\ \ f>0,f \in \B^{+}_b([0,\infty)).$$
(2) Repeat the proof of (1) with $\xi(t)=0$ and $\tau=\infty$. From \eqref{3.5}, we arrive at
\begin{equation}\begin{split}\label{3.5'}
&\e^{(1-h) (\delta -\frac{h}{2}) T} \frac{1}{1-h}(Y_{T}^{1-h}-X_{T}^{1-h})\leq \frac{1}{1-h} (y^{1-h}-x^{1-h}),
\end{split}\end{equation}
which means
\begin{equation}\begin{split}\label{3.5''}
&\rho(X_T^y,X_T^x)\leq \e^{-(1-h) (\delta -\frac{h}{2}) T} \rho(y,x).
\end{split}\end{equation}
Thus, for any $f\in C^1_b([0,\infty)$, we have
\begin{align*}
|\nabla^h P_Tf(x)|&=\lim_{\rho(y,x)\to 0}\frac{|P_Tf(y)-P_Tf(x)|}{\rho(y,x)}\\
&=\lim_{\rho(y,x)\to 0}\frac{|\E f(X_T^y)-\E f(X_T^x)|}{\rho(y,x)}\\
&=\lim_{\rho(y,x)\to 0}\frac{|\E f(X_T^y)-\E f(X_T^x)|}{\rho(X_T^y,X_T^x)}\frac{\rho(X_T^y,X_T^x)}{\rho(y,x)}\\
&\leq \e^{-(1-h) (\delta -\frac{h}{2}) T}P_T |\nabla^ h f|(x).
\end{align*}
\end{proof}
\begin{rem}\label{hge} In \cite{ZZ}, i.e. $h=\frac{1}{2}$, as $\varepsilon$ goes to $0$, the first and second term in $M(X_t,Y_t,\varepsilon)$ can be non-positive if $\alpha\geq \frac{1}{4}$. However, it does not hold when $h\in(\frac{1}{2},1)$, and this is why we construct $M(X_t,Y_t,\varepsilon)$ as in the proof of Theorem \ref{T-Har}.
\end{rem}
\section{Proof of Theorem \ref{T-Super}}
In this section, we use isoperimetric constant to derive the super Poincar\'{e} inequality.
\begin{lem}\label{infty} There exists a small enough constant $r_0\in(0,1)$ such that for any $x_1,x_2>0$  satisfying $\mu((0,x_1))=\mu((x_2,\infty))\leq r_0$, it holds
 $$\mu_{\partial}(\partial (0,x_{1}) ) > \mu_{\partial}(\partial (x_{2},\infty)).$$
\end{lem}
\begin{proof}
By the definition of $\mu_\partial$, we have
\begin{equation}\begin{split}\nonumber
\mu_{\partial}((0,x))
&=\lim\limits_{\varepsilon \to 0} \frac{\mu \left( \{y :0<\frac{1}{1-h}(y^{1-h}-x^{1-h})  \leq \varepsilon \}\right)}{\varepsilon }\\
&=\lim\limits_{\varepsilon \to 0} \frac{\int_{x}^{[x^{1-h}+(1-h) \varepsilon]^{\frac{1}{1-h}}} \eta (y) \d y}{\varepsilon }\\
&=\lim\limits_{\varepsilon \to 0} \frac{\eta (x)\left\{[x^{1-h}+(1-h) \varepsilon]^{\frac{1}{1-h}}-x\right\}}{\varepsilon }\\
&=x^{h} \eta (x)=\frac{\Gamma_0x^{-h}\e^{\frac{2 \alpha }{1-2h} x^{1-2h}-\frac{\delta }{1-h} x^{2-2h}}}{Z}.
\end{split}\end{equation}
Similarly, we arrive at
\begin{equation}\begin{split}\nonumber
\mu_{\partial}((x,\infty))
&=\lim\limits_{\varepsilon \to 0} \frac{\mu \left( \{y :-\varepsilon\leq\frac{1}{1-h}(y^{1-h}-x^{1-h}) <0 \}\right)}{\varepsilon }\\
&=\lim\limits_{\varepsilon \to 0} \frac{\int_{[x^{1-h}-(1-h) \varepsilon]^{\frac{1}{1-h}}}^{x} \eta (y) \d y}{\varepsilon }\\
&=\lim\limits_{\varepsilon \to 0} \frac{\eta (x)\left\{x-[x^{1-h}-(1-h) \varepsilon]^{\frac{1}{1-h}}\right\}}{\varepsilon }\\
&=x^{h} \eta (x)=\frac{\Gamma_0x^{-h}\e^{\frac{2 \alpha }{1-2h} x^{1-2h}-\frac{\delta }{1-h} x^{2-2h}}}{Z}.
\end{split}\end{equation}
Letting $(x^h\eta(x))'=0$, we get
$$2\alpha=2\delta x+h x^{2h-1}.$$
Since $h\in(\frac{1}{2},1)$, there exists $x_{0}$ such that $x^{h} \eta (x)$ is strictly increasing on $(0,x_{0})$ and strictly decreasing on $(x_{0},\infty)$.

Letting $r>0$ be small enough, take $x_{1}(r), x_{2}(r)\in[0,\infty)$ such that $$\mu ((0,x_{1}(r)))=\mu ((x_{2}(r),\infty))=r.$$
It is clear \begin{align}\label{0i}\lim_{r \to 0}x_{1} (r)=0, \ \ \lim_{r \to 0}x_{2} (r)=\infty.
\end{align}
Moreover, by L'Hopital's rule, we have $$\lim\limits_{r \to 0}\frac{\int_{0}^{x_{1}(r)}s^{-2h}\e^{\frac{2 \alpha }{1-2h} s^{1-2h}-\frac{\delta }{1-h} s^{2-2h}}\d s}{\frac{1}{2\alpha}\e^{\frac{2 \alpha }{1-2h} (x_{1}(r))^{1-2h}}}=1,$$
and
\begin{align}\label{luo}
\lim\limits_{r \to 0}\frac{\int_{x_{2}(r)}^{\infty}s^{-2h}\e^{\frac{2 \alpha }{1-2h} s^{1-2h}-\frac{\delta }{1-h} s^{2-2h}}\d s}{\frac{1}{2 \delta}(x_{2}(r))^{-1}\e^{-\frac{\delta }{1-h} (x_{2}(r))^{2-2h}}}=1.
\end{align}
Thus, it holds
\begin{align*}
&1=\lim\limits_{r \to 0}\frac{\frac{1}{2\alpha}\e^{\frac{2 \alpha }{1-2h} (x_{1}(r))^{1-2h}}}{\frac{1}{2 \delta}(x_{2}(r))^{-1}\e^{-\frac{\delta }{1-h} (x_{2}(r))^{2-2h}}}\\
&=\frac{\delta}{\alpha}\lim_{r \to 0}\frac{\e^{\frac{2 \alpha }{1-2h} (x_{1}(r)^{-1})^{2h-1}}}{(x_{2}(r))^{-1}\e^{-\frac{\delta }{1-h} (x_{2}(r))^{2-2h}}}\\
&=\frac{\delta}{\alpha}\lim_{r \to 0}\e^{-\frac{2 \alpha }{2h-1} (x_{1}(r)^{-1})^{2h-1}+\frac{\delta }{1-h} (x_{2}(r))^{2-2h}+\log (x_{2}(r))}.
\end{align*}
This means $$\lim_{r \to 0}\left\{-\frac{2 \alpha }{2h-1} (x_{1}(r)^{-1})^{2h-1}+\frac{\delta }{1-h} (x_{2}(r))^{2-2h}+\log (x_{2}(r))\right\}=\log\frac{\alpha}{\delta}.$$
Thus, \eqref{0i} yields
$$\lim_{r \to 0}\left\{-\frac{2 \alpha }{2h-1} \frac{(x_{1}(r)^{-1})^{2h-1}}{(x_{2}(r))^{\frac{(1-h)(2h-1)}{h}}}+\frac{\delta }{1-h} \frac{(x_{2}(r))^{2-2h}}{(x_{2}(r))^{\frac{(1-h)(2h-1)}{h}}}\right\}=0.$$
Since
$$\lim_{r\to0}\frac{(x_{2}(r))^{2-2h}}{(x_{2}(r))^{\frac{(1-h)(2h-1)}{h}}} =\lim_{r\to0}(x_{2}(r))^{\frac{1-h}{h}}=\infty,$$
we have
$$\lim_{r \to 0}\frac{(x_{1}(r)^{-1})^{2h-1}}{(x_{2}(r))^{\frac{(1-h)(2h-1)}{h}}}=\infty.$$
This together with \eqref{0i} and the representation of $\mu_\partial((0,x))$ and $\mu_\partial((x,\infty))$ implies
\begin{align*}
&\lim\limits_{r \to 0}\frac{\mu_{\partial}(\partial (0,x_{1}(r)))}{\mu_{\partial}(\partial (x_{2}(r),\infty))}\\
&=\lim\limits_{r \to 0}\frac{(x_{1}(r))^{-h}\e^{\frac{2 \alpha }{1-2h} (x_{1}(r))^{1-2h}-\frac{\delta }{1-h} (x_{1}(r))^{2-2h}}}{(x_{2}(r))^{-h}\e^{\frac{2 \alpha }{1-2h} (x_{2}(r))^{1-2h}-\frac{\delta }{1-h} (x_{2}(r))^{2-2h}}}\\
&=\lim\limits_{r \to 0}\frac{(x_{1}(r))^{-h}\e^{\frac{2 \alpha }{1-2h} (x_{1}(r))^{1-2h}}}{(x_{2}(r))^{-h}\e^{-\frac{\delta }{1-h} (x_{2}(r))^{2-2h}}}\\
&=\frac{\alpha}{\delta}\lim\limits_{r \to 0}\frac{((x_{1}(r))^{-1})^{h}}{(x_{2}(r))^{1-h}}\\
&=\frac{\alpha}{\delta}\lim_{r \to 0}\left(\frac{(x_{1}(r)^{-1})^{2h-1}}{(x_{2}(r))^{\frac{(1-h)(2h-1)}{h}}}\right)^{\frac{h}{2h-1}}=\infty.
\end{align*}
So, there exists $r_{0}>0$ such that for any $x_{1}, x_{2}\in[0,\infty)$ satisfying $\mu ((0,x_{1}))=\mu ((x_{2},\infty))\leq r_{0}$, it holds
\begin{align}\label{com}
\mu_{\partial}(\partial (0,x_{1}) ) > \mu_{\partial}(\partial (x_{2},\infty)).
\end{align}
Thus, we complete the proof.
\end{proof}
\begin{proof}[Proof of Theorem \ref{T-Super}]
(1) Firstly, we prove that there exists small enough $\bar{r}_0>0$ such that for any $r\in(0,\bar{r}_0)$, $k(r)$ can only get the lower bound on the set $(x,\infty)$ with $\mu((x,\infty))\leq r$.

Let $\bar{r}_0=\frac{1}{2}\{\mu(0,x_0)\wedge\mu(x_0,\infty)\}\wedge r_0$ with $r_0$ introduced in Lemma \ref{infty}. Fix $r\in(0,\bar{r}_0)$. For any open set $A\subset [0,\infty)$ with $\mu(A)=r$, let $A_{1}:=A \cap (0,x_{0})$ and $A_{2}:=A \cap (x_{0},\infty)$. Then $\mu(A_1)\leq \frac{1}{2}\mu(0,x_0)$ and $\mu(A_2)\leq \frac{1}{2}\mu(x_0,\infty)$. 
Let $x_2=\inf\{x:x\in A_2\}$ and $x_1=\sup\{x:x\in A_1\}$.
Take $\bar{x}_1\leq x_1$ and $\bar{x}_2\geq x_2$ such that $ \mu((0,\bar{x}_{1}))=\mu( A_{1})$ and $\mu( (\bar{x}_{2},\infty))=\mu(A_2)$.
Since $x^{h} \eta (x)$ is strictly increasing on $(0,x_{0})$ and strictly decreasing on $(x_{0},\infty)$,
we have $$\mu_{\partial}(\partial A_{1} )\geq \mu_{\partial}(\partial (0,x_{1}))\geq\mu_{\partial}(\partial (0,\bar{x}_{1}))$$
and $$\mu_{\partial}(\partial A_{2} )\geq \mu_{\partial}(\partial (x_{2},\infty))\geq\mu_{\partial}(\partial (\bar{x}_{2},\infty)).$$
This yields $$\frac{\mu_\partial(\partial A)}{\mu(A)}\geq\frac{\mu_{\partial}(\partial ((0,\bar{x}_{1})\cup (\bar{x}_{2},\infty)) )}{\mu((0,\bar{x}_{1})\cup (\bar{x}_{2},\infty))}. $$
For any $y_1, y_2\in(0,\infty)$ satisfying $\mu((0,y_{1}))+\mu((y_{2},\infty))=r$, define $$\varphi (y_{1},y_{2}):=\mu_{\partial}(\partial ((0,y_{1})\cup (y_{2},\infty)) )=y_1^h\eta(y_1)+y^h_2\eta(y_2).$$
Next, we show that $$\varphi (0,x)=\inf\{\varphi (y_{1},y_{2}):\mu((0,y_{1}))+\mu((y_{2},\infty))=r\}, $$
here, $\mu((x,\infty))=r$.
In fact, from $\mu((0,y_{1}))+\mu((y_{2},\infty))=r$, there exists a function $\phi$ such that $y_2=\phi(y_1)$ and $\phi'(y_1)=\frac{\eta(y_1)}{\eta(y_2)}$. Thus, we obtain
$$\varphi(y_1, y_2)=\varphi(y_1, \phi(y_1))=:\Phi(y_1). $$
By the representation of $\eta(s)$, we have
\begin{align*}
\Phi'(y_1)&=\frac{\partial\varphi}{\partial y_1}(y_1, y_2)+\frac{\partial\varphi}{\partial y_2}(y_1, y_2)\phi'(y_1)\\
&=y_1^h\eta'(y_1)+hy_1^{h-1}\eta(y_1)+(y_2^h\eta'(y_2) +hy_2^{h-1}\eta(y_2))\phi'(y_1)\\
&=y_1^h\eta'(y_1)+hy_1^{h-1}\eta(y_1)+(y_2^h\eta'(y_2) +hy_2^{h-1}\eta(y_2))\frac{\eta(y_1)}{\eta(y_2)}\\
&=\frac{\Gamma_0}{Z}\e^{\frac{2 \alpha }{1-2h} y_1^{1-2h}-\frac{\delta }{1-h} y_1^{2-2h}}\Bigg(y_1^h(-2hy_1^{-2h-1}+y_1^{-2h}2\alpha y_1^{-2h}-y_1^{-2h}2\delta y_1^{1-2h})+hy_1^{h-1}y_1^{-2h}\\
&+y_2^h\frac{y_1^{-2h}}{y_2^{-2h}}(-2hy_2^{-2h-1}+y_2^{-2h}2\alpha y_2^{-2h}-y_2^{-2h}2\delta y_2^{1-2h})+hy_2^{h-1}y_1^{-2h}\Bigg)\\
&=\eta(y_1)\Bigg(-h(y_1^{h-1}+y_2^{h-1})+2\alpha (y_1^{-h}+y_2^{-h})-2\delta (y_1^{1-h}+y_2^{1-h})\Bigg).
\end{align*}
Since $h\in(\frac{1}{2},1)$ and $\alpha,\delta>0$, there exists a small enough constant $r_1>0$ such that $\Phi'(y_1)>0$ when $y_1\in(0,r_1)$, and there exists a big enough constant $r_2>0$ such that $\Phi'(y_1)<0$ when $y_2\in(r_2,\infty)$. Thus, $\varphi$ can only take minimum on $(y_2,\infty)$ with $\mu((y_2,\infty))=r$ or on $(0,y_1)$ with $\mu((0,y_1))=r$. By \eqref{com}, $\varphi$ take minimum on $(y_2,\infty)$ with $\mu((y_2,\infty))=r$.
Thus, we obtain $$\frac{\mu_\partial(\partial A)}{\mu(A)}\geq \frac{\mu_{\partial}(\partial ((0,\bar{x}_{1})\cup (\bar{x}_{2},\infty)) )}{\mu((0,\bar{x}_{1})\cup (\bar{x}_{2},\infty))} \geq \frac{\mu_{\partial}(\partial (x,\infty))}{\mu((x,\infty))},$$
here, $\mu((x,\infty))=r$.
Thus, we have
$$k(r)=\inf\limits_{\{x|\mu((x,\infty)) \leq r\}} \frac{x^{h}\eta(x)}{\mu((x,\infty))}.$$
Take $x_r>0$ such that $\mu((x_{r},\infty))=r$. Then we have $\lim_{r\to0}x_r=\infty$. By \eqref{luo}, we have \begin{align*}
\lim\limits_{r\to0}\frac{x_r^{2h-1}\eta(x_r)}{\mu((x_r,\infty))}=2 \delta.
 \end{align*}
This implies
\begin{align}\label{inf}
\lim_{r\to 0}k(r) =\lim_{r\to0}\frac{x_r^{h}\eta(x_r)}{\mu((x_r,\infty))} =\lim_{r\to0}\frac{x_r^{1-h}x_r^{2h-1}\eta(x_r)}{\mu((x_r,\infty))}=\lim_{r\to0}x_r^{1-h}=\infty.
\end{align}

According to \cite[Theorem 3.4.16]{W2}, the super Poincar\'{e} inequality holds for $$\beta (r)=\frac{4}{k^{-1}(2\sqrt{2}r^{-\frac{1}{2}})},\ \  r>0.$$

(2) It follows form \eqref{luo} that
$$\lim_{r \to 0}\frac{\mu((x_{r},\infty))}{\frac{\Gamma_0x_{r}^{-1}\e^{\frac{2 \alpha }{1-2h} x_{r}^{1-2h}-\frac{\delta }{1-h} x_{r}^{2-2h}}}{2Z\delta}}=1,$$
which implies
$$\lim_{r\to0}\frac{\e ^{\log r}}{\e^{-\frac{\delta }{1-h} x_{r}^{2-2h}-\log x_{r}}}=\lim_{r\to0}\e ^{\log r+\frac{\delta }{1-h} x_{r}^{2-2h}+\log x_{r}}=\frac{\Gamma_0}{2Z\delta}.$$
This yields
$$\lim_{r\to0}\{\log r+\frac{\delta }{1-h} x_{r}^{2-2h}+\log x_{r}\}=\log\frac{\Gamma_0}{2Z\delta}.$$
Since $h\in(\frac{1}{2},1)$, we obtain
$$\lim_{r\to0}\frac{\sqrt{\log r^{-1}}}{x_{r}^{1-h}}=\sqrt{\frac{\delta }{1-h} }.$$
Combining this with \eqref{inf}, we arrive at
$$\lim_{r\to 0}\frac{k(r)}{\sqrt{\frac{1-h }{\delta} }\sqrt{\log r^{-1}}} =\lim_{r\to0}\frac{x_r^{1-h}}{\sqrt{\frac{1-h }{\delta} }\sqrt{\log r^{-1}}}=1.$$
Thus, there exist constants $r_0>0$ and $c>0$ such that $k(r)\geq c[-\log r]^{\frac{1}{2}}$ for $r\in(0.r_0)$. According to \cite[Corollary 3.4.17]{W2} with $\delta=1$, \eqref{sp} holds with $\beta (r)= \e^{C(1+r^{-1})}$ for some constant $C>0$.

(3) Let $\rho(0,x)=\frac{1}{1-h}x^{1-h}$, then $\rho(0,\cdot) \in \mathcal{D}(\mathcal{E})$.  Set $h_n=\rho(0,\cdot)\wedge n$. For any $g\in  \mathcal{D}(\mathcal{E})$ with $\mu(|g|)\leq 1$, we have \begin{align*}
&\mathcal{E}(h_{n}g,h_{n})-\frac{1}{2}\mathcal{E}(h^{2}_{n},g)\\
&=\frac{1}{2}\int_{0}^{\infty} x^{2h}(h_{n}g)^{'}(x)h_{n}^{'}(x)\mu (\mathrm{d}x)-\frac{1}{4}\int_{0}^{\infty} x^{2h}(h^{2}_{n})'(x)g'(x)\mu (\mathrm{d}x)\\
&\leq\frac{1}{2}\int_{0}^{(n(1-h))^{\frac{1}{1-h}}} x^{2h}(h_{n}^{'})^2(x)g(x)\mu (\mathrm{d}x)\leq \frac{1}{2}\mu (|g|) \leq \frac{1}{2}.
 \end{align*}
So by \cite[Definition 1.2.1]{W2}, $L_{\mathcal{E}}(\rho(0,\cdot))\leq 1.$

However, for any $\lambda\in(\frac{1}{2},1)$ and $\varepsilon>0$, we have $$\mu \{\exp\{\varepsilon \rho(0,\cdot)^{\frac{2\lambda}{2\lambda-1}}\}\}=\frac{\Gamma_0}{Z}\int_{0}^{\infty}x^{-2h}\e^{\frac{2 \alpha }{1-2h} x^{1-2h}-\frac{\delta }{1-h} x^{2-2h}+\varepsilon(\frac{1}{1-h})^{\frac{2\lambda}{2\lambda-1}}x^{\frac{2\lambda(1-h)}{2\lambda-1}} }=\infty, $$ here, in the last display, we have used $\frac{2\lambda}{2\lambda-1}>2$ for any $\lambda\in(\frac{1}{2},1)$.
By \cite[Corollary 3.3.22]{W2}, the super Poincar\'{e} inequality \eqref{sp} does not hold with $\beta(r)=\e^{C(1+r^{-\lambda})}$ for $\frac{1}{2}<\lambda<1$. Similarly, we can show $\mu(\exp[\exp(\varepsilon \rho(0,\cdot))])=\|\rho(0,\cdot)\|_\infty=\infty$. Again by \cite[Corollary 3.3.22]{W2}, \eqref{sp} does not hold with $\beta(r)=\e^{C(1+r^{-\lambda})}$ for $0<\lambda\leq \frac{1}{2}$. Thus, we finish the proof.
\end{proof}
\paragraph{Acknowledgement.} The authors would like to thank Professor Feng-Yu Wang and Shao-Qin Zhang for corrections and helpful comments.

\beg{thebibliography}{99}

\bibitem{A} Alfonsi, A.,  \emph{On the discretization schemes for the CIR (and Bessel squared) processes,} Monte Carlo Methods and Applications, 11(2005), 355-384.

\bibitem{A1} Alfonsi, A.,  \emph{High order discretization schemes for the CIR process: Application to affine term structure and Heston models,} Mathematics of Computation, 79(2010), 209-237.

\bibitem{A2} Alfonsi, A.,  \emph{Strong order one convergence of a drift implicit Euler scheme: Application to the CIR process,} Statistics and Probability Letters, 83(2013), 602-607.

\bibitem{ATW} Arnaudon, M., Thalmaier, A., Wang, F.-Y.,   \emph{Harnack inequality and heat kernel estimates on manifolds with curvature unbounded below,} Bull. Sci. Math., 130(2006), 223-233.

\bibitem{C} Cairns, A. J. G., \emph{Interest rate models: an introduction,} Princeton University
Press, 2004.

\bibitem{CJM} Chassagneux, J. F., Jacquler, A., Mihaylov, I.,  \emph{An Explicit Euler Scheme with Strong Rate of Convergence for Financial SDEs with Non-Lipschitz Coefficients,} Siam Journal on Financial Mathematics, 7(2016), 993-1021.

\bibitem{CIR} Cox, J. C., Ingersoll, J. E., Ross, S. A.,  \emph{A theory of the term structure of interest rates,} Econometrica, 53(1985), 385-407.

\bibitem{CIR1} Cox, J. C., Ross, S. A.,  \emph{An intertemporal general equilibrium model of asset prices,} Econometrica, 53(1985), 363-384.

\bibitem{CL} Chou, C. S., Lin, H. J.,  \emph{Some properties of CIR processes,} Stochastic Analysis and Applications, 24(2006), 901-912.

 \bibitem{GR} Gy\"ongy, I., R\'{a}sonyi, M., A note on Euler approximations for SDEs with H\"older continuous diffusion coefficients, Stoch. Proc. Appl., 121(2011), 2189--2200.

\bibitem{IW} Ikeda, N., Watanabe, S.,  \emph{Stochastic differential equations and diffusion processes,} 2nd ed. Amsterdam: North Holland, 1989.

\bibitem{KS} Karatzas, I., Shreve, S. E., \emph{Brownian motion and stochastic calculus,} 2nd
edition, corrected 6th printing. Springer, 2000.



\bibitem{S} Stamatiou, I., S.,  \emph{An explicit positivity preserving numerical scheme for CIR/CEV type delay models with jump,} Journal of Computational and Applied Mathematics, 360(2019), 78-98.

\bibitem{WMC} Wu, F., Mao, X., R., Chen, k.,  \emph{The Cox-Ingersoll-Ross model with delay and strong convergence of its Euler-Maruyama approximate solutions,} Applied Numerical Mathematics, 59(2009), 2641-2658.

\bibitem{W2} Wand, F.-Y.,  \emph{Functional inequalities, Markov semigroups and Spectral theory,} Beijing: Science Press, 2005.

\bibitem{W} Wang, F.-Y.,  \emph{Harnack Inequality and Applications for Stochastic Partial Differential Equations,} Berlin: Springer, 2013.

\bibitem{W3} Wang, F.-Y.,  \emph{Harnack inequality and applications for stochastic generalized porous media equations,} Ann Probab., 35(2007), 1333-1350.

\bibitem{W4} Wang, F.-Y.,  \emph{Harnack inequality on manifolds with boundary and applications[J].} J. Math. Pures Appl., 94(2010), 304-321.

\bibitem{WY} Wang, F.-Y., Yuan, C.,  \emph{Harnack inequality for functional SDEs with multiplicative noise and applications,} Stoch. Proc. Appl., 121(2011), 2692-2710.

\bibitem{YW} Yang, X., Wang, X., J.,  \emph{A transformed jump-adapted backward Euler method for jump-extended CIR and CEV models,} Numerical Algorithms, 74(2017), 39-57.

\bibitem{ZZ} Zhang, S.-Q., Zheng, Y.,  \emph{Functional Inequality and Spectrum Structure for CIR Model, in Chinese,} Journal of  Beijing Normal University (Natural Science), 54(2018), 572-582.
\end{thebibliography}

\end{document}